\documentclass{article}

\usepackage{amsmath}
\usepackage{amsfonts}
\usepackage{amsthm}
\usepackage{amssymb}
\usepackage{graphicx}
\usepackage{color}
\usepackage{verbatim} 
\usepackage{hyperref}
\usepackage{enumerate}
\usepackage{float}
\usepackage{tabulary}
\usepackage[utf8]{inputenc}
\usepackage{tikz-cd}

\usepackage{fullpage}

\theoremstyle{plain}
\newtheorem{theorem}{Theorem}

\newtheorem{corollary}[theorem]{Corollary}
\newtheorem{lemma}[theorem]{Lemma}
\newtheorem{prop}[theorem]{Proposition}

\theoremstyle{definition}

\newtheorem{definition}{Definition}
\newtheorem{example}{Example}

\title{Discreet Coin Weighings and the Sorting Strategy}
\author{Tanya Khovanova \and Rafael M. Saavedra}
\date{}

\begin{document}

\maketitle

\begin{abstract}
In 2007, Alexander Shapovalov posed an old twist on the classical coin weighing problem by asking for strategies that manage to conceal the identities of specific coins while providing general information on the number of fake coins. In 2015, Diaco and Khovanova studied various cases of these ``discreet strategies" and introduced the revealing factor, a measure of the information that is revealed.

In this paper we discuss a natural coin weighing strategy which we call the sorting strategy: divide the coins into equal piles and sort them by weight. We study the instances when the strategy is discreet, and given an outcome of the sorting strategy, the possible number of fake coins. We prove that in many cases, the number of fake coins can be any value in an arithmetic progression whose length depends linearly on the number of coins in each pile. We also show the strategy can be discreet when the number of fake coins is any value within an arithmetic subsequence whose length also depends linearly on the number of coins in each pile. We arrive at these results by connecting our work to the classic Frobenius coin problem. In addition, we calculate the revealing factor for the sorting strategy.
\end{abstract}

\section{Introduction}

In 2007, Alexander Shapovalov posed a novel variation on the traditional coin weighing problem for the International Kolmogorov Math Tournament~\cite{ShapovalovProblem}. In traditional coin weighing problems, one is tasked with identifying fake coins of lesser weight amongst a collection of indistinguishable coins through the use of a balance scale, minimizing the number of instances that the balance scale is employed. However, in Shapovalov's problem, which we call the \emph{discreet coin weighing problem}, a lawyer with full knowledge of the identities of the coins must convey information to an observer through a demonstration with a balance scale about the number of fake coins and yet simultaneously keep secret whether any specific coins are fake or not. Such a sequence of weighings is said to be a \emph{discreet weighing strategy}.

Discreet strategies were first investigated by Knop~\cite{Knop} and then by Diaco and Khovanova~\cite{Diaco-Khovanova}. In their paper, Diaco and Khovanova studied, for various triples of integers $(t,f,d)$, discreet weighing strategies that demonstrate that among $t$ total coins, the number of fake coins is not $d$ and may be $f$. Diaco and Khovanova showed that for certain $(t,f,d)$, there exist no discreet weighing strategies. Additionally,~\cite{Diaco-Khovanova} introduced the \emph{revealing factor}, a measure of how much information a weighing strategy reveals to an observer apart from the information that it is designed to reveal, and investigated weighing strategies with minimal revealing factor. A reduced version of~\cite{Diaco-Khovanova} was published in~\cite{Diaco-KhovanovaPaper}. In~\cite{DiacoPreprint}, Diaco employed a new formalism and notation, extended some results of~\cite{Diaco-Khovanova}, and discussed related coin problems.

The central attention of the paper is a class of weighings which we call the \emph{sorting strategy}. The sorting strategy consists of dividing all $t$ coins into $p$ piles of equal size and totally ordering all piles by weight. We describe the order relations of the piles by a \emph{sorting sequence}. In our work, we shift our attention from attempting to simply show, for a given $d$ and $f$, the impossibility of $d$ and the possibility of $f$ for the number of fake coins as was done in~\cite{Diaco-Khovanova,DiacoPreprint}. We analyze all possible numbers of fake coins consistent with a given sorting sequence, and for which values the strategy is discreet. We also calculate the revealing factor.

Our main results are that when there are sufficiently many coins, the possible number of fake coins can be any value within an arithmetic progression, and that for an arithmetic subsequence of that progression, the strategy is discreet. These arithmetic progressions increase linearly with the sizes of the piles. Consequently, the possible set of values for the number of fake coins can be divided into three parts, in which the first and third parts consist of fixed quantities of exceptional values, while the central set is an arithmetic progression that increases when the total number of coins increases. To prove the existence of this arithmetic progression of possible values, we employ the theory of the \emph{Frobenius coin problem} as well as a \emph{redistribution procedure} for the fake coins.

We begin our paper by clarifying the concept of discreetness in Section~\ref{sec:preliminaries} and reviewing the basic definitions and conventions for the discreet coin weighing problem. Because the object of our work is to find all possible values for the number of fake coins, we emphasize the dependence of the concept of discreetness on the value being proved.

In Section~\ref{sec:sorting} we develop the fundamentals to study the sorting strategy. We introduce the crucial difference between general solutions and those that also satisfy a specified height bound to model the sorting strategy algebraically. We prove necessary and sufficient conditions on when the sorting strategy is discreet. One can find several illustrative examples as well as some numerical data on the values for the possible number of fake coins provided in Table~\ref{table:sorting-sequence}. We finish by calculating the minimum and maximum values for the number of fake coins for a given sorting sequence and the minimum and maximum values for which it is discreet.

The algebraic treatment of the solutions for the sorting sequence is developed further in Section~\ref{sec:duality-frobenius}. A basic introduction to the \emph{Frobenius coin problem}, a classic topic of research in combinatorics, is given in subsection~\ref{sec:frobenius}, and we explain its connection to the sorting strategy.

In Section~\ref{sec:redistribution} we utilize combinatorial methods of coin redistribution to show conditions that allow us to deal with the height bound and thus prove that there exist configurations with a certain number of fake coins.

We prove our central results on the existence of arithmetic progressions for the sorting strategy in Section~\ref{sec:arith-progressions}. We also utilize work on the Frobenius problem to calculate bounds on the endpoints of the arithmetic progressions.

Finally, in Section~\ref{sec:revealing-factor}, we provide a formula for the revealing factor for the sorting strategy.

Our main results demonstrate the resilience of the sorting strategy as a discreet weighing strategy. This constitutes an extension of the previous work on discreet weighing strategies in that we show not only how a strategy may protect the privacy of coins even when the number of fake coins is known, but also how the very number of fake coins may be obscured amongst an arbitrarily large range of values.

\section{Preliminaries}\label{sec:preliminaries}

In this section we establish the definitions and conventions employed in this paper.

We have a finite number of coins. Any real coins all weigh the same and any fake coins all weigh the same. Fake coins are lighter than real coins. All coins are outwardly indistinguishable but have been labeled so they may be kept track of. At our disposal is a balance scale on which we may place equal number of coins on each side, and which will indicate which side weighs less, or if they weigh the same.

A \emph{weighing strategy} is a specification of a sequence of weighings based on the labels of the coins.

\begin{definition}
A weighing strategy \emph{discreetly} proves a certain property $P$ if given an outcome for the weighing strategy and the information that $P$ is true, the identity of no specific coin can be concluded. We say that such a strategy is \emph{discreet}.
\end{definition}

In this paper we will consider, for a given weighing strategy, several possibilities for the number of fake coins. In~\cite{Diaco-Khovanova, DiacoPreprint}, the only aspects of a weighing strategy that were considered were, for a single triple $(t,f,d)$ whether there could be $f$ fake coins as opposed to $d$ fake coins out of $t$ total coins, even if there could also be $f'$ fake coins. The new definition makes it clear that discreetness for a given strategy depends on the specific property that is being proved, such as the various values for the number of fake coins that are possible.

In~\cite{Diaco-Khovanova}, the revealing factor was introduced to measure how much information a weighing strategy $\mathcal A$ revealed in the course of proving a property $P$ in addition to what knowing $P$ necessarily reveals. If $\mathcal A$ is a weighing strategy and $P$ a property about coins, then we let the \emph{old possibilities} be the set of coin configurations for which $P$ is true, and the \emph{new possibilities} be the set of coin configurations consistent with the outcome of $\mathcal A$ for which $P$ is also true.

\begin{definition}
Let $\mathcal A$ be a weighing strategy and $P$ a property of coins. If the number of new possibilities is not zero, then the \emph{revealing factor} is the ratio of the number of old to new possibilities: \[X = \frac{\# \text{ old possibilities}}{\# \text{ new possibilities}}.\]
\end{definition}

If the objective of a weighing strategy $\mathcal A$ is to prove a property $P$, then the revealing factor will measure the amount of information revealed by $\mathcal A$ relative to what knowing $P$ necessarily reveals.

In general, the revealing factor of a weighing strategy is not closely related to whether it is discreet. Indeed, as is shown in~\cite{Diaco-Khovanova}, sometimes, indiscreet strategies may have lower revealing factors that discreet strategies. However, there exists an upper bound for the revealing factor of a discreet strategy that is proven in~\cite{DiacoPreprint}.

\section{The sorting strategy: Several piles of equal size} \label{sec:sorting}

Suppose that we have $t=pk$ total coins, so that we can form $p$ piles of size $k$. The \emph{sorting strategy} is the weighing strategy where we compare all the piles in order to sort them in order of relative weight.

To describe the possible outcomes of the sorting strategy, that is, the relations between the piles, we mark the heaviest piles with 0, the second heaviest with 1 and so on, noting that the lighter the pile, the more fake coins it will contain. The relations of relative weight between the piles become encoded in a non-decreasing sequence of integers. We can formalize these sequences as follows:

\begin{definition}
A \textit{sorting sequence} is a non-decreasing sequence of finite length of non-negative integers, beginning with 0, in which each entry is equal to the previous or is greater by 1.
\end{definition}

We let $S_i$ be the $i$th sorting sequence of length $p$ under the lexicographic ordering, beginning with index~$0$.

\begin{lemma}
The number of sorting sequences of length $p$ is $2^{p-1}$.
\end{lemma}

\begin{proof}
The sequence always starts with zero. After the first entry, each can either be the same or increase by one. This choice is made $p-1$ times. The number of such sequences is thus $2^{p-1}$.
\end{proof}

\begin{definition}
The \textit{binary representation} of $S_i$ of length $p$ is a binary string of length $p-1$ constructed from $S_i$ in which the $j$th digit is 0 if the $(j+1)$-st and $j$th entries of $S_i$ are the same, and 1 if they are different.
\end{definition}

We let $B_i$ be a non-negative integer $i < 2^{p-1}$ in binary padded with zeros from the left to make a string of length $p-1$.  Additionally, if $s$ is a binary sequence, then we denote the reverse sequence by $s'$.

\begin{prop}
The binary representation of the sequence $S_i$ is $B_i$.
\end{prop}
\begin{proof}
First, we prove that if $a < b$, then the binary representation of $S_a$ precedes the binary representation of $S_b$. If $a<b$, then there is a smallest integer $j$ such that the $j$th entry of $S_a$ is smaller than the $j$th entry of $S_b$. Then the binary representations of $S_a$ and $S_b$ coincide until the $(j-1)$ entry is 0 for $S_a$ and 1 for $S_b$. Consequently, the binary representation of $S_a$ precedes the binary representation of $S_b$.

Consequently the map from $S_i$ to its binary representation is injective. Because the set of sorting sequences of length $p$ and the set of binary sequences of length $p$ have the same size, the map is bijective and the equality must hold for all $i$.
\end{proof}

We let $S_i'$ be the sorting sequence whose binary representation is the reverse of the binary representation of $S_i$.

Observe how the relations of relative weight between the piles in the sorting strategy are encoded in a sorting sequence:

\begin{example}\label{ex:sorting-sequence}
Suppose that $p=4$ and the sorting sequence is $S_3=(0,0,1,2)$. The binary representation of $S_3$ is $B_3=011$. We have four piles of equal size, $P_1, P_2, P_3, P_4$, with the following relations of relative weight according to the sequence: $P_1 = P_2 > P_3 > P_4$, so that $P_1$ and $P_2$ have the fewest fake coins, and $P_4$ the most.
\end{example}

We can use sorting sequences to algebraically describe the pile relations. Given the sorting sequence $S_i$, let $p_j$ be the number of $j$'s in $S_i$ so that $p_1 + \dotsb + p_r = p$, the total number of piles. Let $f_j$ be a possible number of fake coins in the piles corresponding to the $j$'s in the sorting sequence. Therefore, the possible total number of fake coins is 
\begin{equation}\label{partition-equation}
p_1 f_1 + \dotsb + p_r f_r = f \quad \quad \text{ with \quad $0 \leq f_1 < \cdots < f_r$}
\end{equation}
with the added conditions that $f_r \leq k$.

In the case when $r=1$, we have $p_1 = p$ and $0 \leq p_1f_1=f \leq pk$; the number of fake coins can be any multiple of $p$ from $0$ to $pk=t$. This trivial case corresponds to the sorting sequence $S_0$. In the rest of the paper we assume that $r>1$.

To analyze how both~\eqref{partition-equation} and the requirement $f_r \leq k$ determine a coin configuration consistent with the sorting strategy and a given sorting sequence, we distinguish the effects of both conditions.

\begin{definition}
A configuration of fake coins $(f_1, \dotsc, f_r,f)$ is a \emph{general solution} for a sorting sequence if it satisfies~\eqref{partition-equation}. A general solution is said to satisfy the \emph{height bound} for $k$ if $f_r \leq k$, so that no pile has more than $k$ fake coins.
\end{definition}

A general solution will correspond to an actual coin configuration if it also satisfies the height bound. When studying the coin configurations consistent with the sorting strategy, we will first find general solutions, and then try to find those that satisfy the height bound for $k$.

\begin{example}
Consider the sorting sequence $(0,0,1,2)$ from Example~\ref{ex:sorting-sequence}. Here, $r=3$ and $(p_1, p_2, p_3)=(2,1,1)$. A general solution is $(f_1, f_2, f_3)=(2,3,4)$. This solution does not respect the height bound for $k=2$ since there are $4$ fake coins in the last pile, but clearly does so for $k \geq 4$. In those cases, $11 = 2+2+3+4$ fake coins are possible.

Note that when $k=1$, there are no solutions as the lightest pile has to have at least $2$ fake coins. For $k=2$, the only solution is $(0,1,2)$, so that the sorting strategy proves that there are $f=0+0+1+2=3$ fake coins. However, we know that all of the last pile is composed of fake coins and all of the first two piles are composed of the real coins, so in this case the sorting strategy is indiscreet.
\end{example}

For a certain $k$, the sorting sequence may prove that certain values for $f$ are possible, and in some cases, may do so discreetly. In Table~\ref{table:sorting-sequence}, we list all the possible values for $f$ for all sorting sequences when $p=5, k=5$. The values which the sorting strategy can prove discreetly are in bold.

\begin{table}[ht] 
\[
\begin{array}{ccc}
 \text{Binary rep} & \text{Sorting sequence} &\text{Values of $f$} \\ \hline
  0000 & 0,0,0,0,0 & 0,\mathbf{5},\mathbf{10},\mathbf{15},\mathbf{20},25 \\
 0001 & 0,0,0,0,1 & 1,2,3,4,5,\mathbf{6},\mathbf{7},\mathbf{8},9,\mathbf{11},\mathbf{12},13,\mathbf{16},17,21 \\
 0010 & 0,0,0,1,1 & 2,4,6,\mathbf{7},8,\mathbf{9},10,\mathbf{11},\mathbf{12},13,\mathbf{14},16,\mathbf{17},19,22 \\
 0011 & 0,0,0,1,2 & 3,4,5,6,7,\mathbf{8},\mathbf{9},\mathbf{10},11,12,\mathbf{13},14,15,18 \\
 0100 & 0,0,1,1,1 & 3,6,\mathbf{8},9,\mathbf{11},12,\mathbf{13},\mathbf{14},15,\mathbf{16},17,\mathbf{18},19,21,23 \\
 0101 & 0,0,1,1,2 & 4,5,6,7,8,\mathbf{9},\mathbf{10},11,\mathbf{12},13,\mathbf{14},15,17,19 \\
 0110 & 0,0,1,2,2 & 5,7,8,9,\mathbf{10},11,\mathbf{12},\mathbf{13},14,\mathbf{15},16,17,18,20 \\
 0111 & 0,0,1,2,3 & 6,7,8,9,10,\mathbf{11},12,13,14,16 \\
 1000 & 0,1,1,1,1 & 4,8,\mathbf{9},12,\mathbf{13},\mathbf{14},16,\mathbf{17},\mathbf{18},\mathbf{19},20,21,22,23,24 \\
 1001 & 0,1,1,1,2 & 5,6,7,8,9,\mathbf{10},\mathbf{11},12,13,\mathbf{14},\mathbf{15},16,17,18,19,20 \\
 1010 & 0,1,1,2,2 & 6,8,10,\mathbf{11},12,\mathbf{13},14,\mathbf{15},\mathbf{16},17,18,19,20,21 \\
 1011 & 0,1,1,2,3 & 7,8,9,10,11,\mathbf{12},13,14,15,16,17 \\
 1100 & 0,1,2,2,2 & 7,10,11,\mathbf{12},13,14,\mathbf{15},\mathbf{16},\mathbf{17},18,19,20,21,22 \\
 1101 & 0,1,2,2,3 & 8,9,10,11,12,\mathbf{13},14,15,16,17,18 \\
 1110 & 0,1,2,3,3 & 9,11,12,13,\mathbf{14},15,16,17,18,19 \\
 1111 & 0,1,2,3,4 & 10,11,12,13,14,15 \\
\end{array}
\]\caption{Values of $f$ for sorting sequences with $p=5$, $k=5$.}
\label{table:sorting-sequence}
\end{table}

Due to the opposite nature of real and fake coins, there is a duality between the general solutions that satisfy height bounds and their reversals.

\begin{lemma}\label{thm:sorting-duality}
If the sorting sequence is $S_i$, then we can have a general solution with $f$ fake coins that satisfies the height bound for $k$ if and only if there exists a general solution for the sorting sequence $S_i'$ with $pk-f$ fake coins that satisfies the height bound.
\end{lemma}

\begin{proof}
A general solution for a sorting sequence $S_i$ satisfying the height bound for $k$ has $f$ fake and $pk-f$ real coins. By exchanging the real and fake coins, it is naturally dual to a general solution for $S_i'$ satisfying the height bound with $pk-f$ fake coins and $f$ real coins.
\end{proof}

\begin{corollary}\label{thm:discreetness-duality}
If the sorting sequence is $S_i$, then the sorting strategy can discreetly prove that there are $f$ fake coins if and only if the sorting strategy can discreetly prove that there can be $pk-f$ fake coins for $S_i'$.
\end{corollary}

\begin{proof}
The sorting strategy can discreetly prove that there are $f$ fake coins if and only if there exist two configurations with $f$ coins where in the first, there is a fake coin in every pile, and in the second, there is a real coin in every pile. If this is the case for $S_i$, then by the inversion of fake and real coins in the proof of Lemma~\ref{thm:sorting-duality}, the sorting strategy can discreetly show that there can be $pk-f$ coins for $S_i'$. The converse follows from the application of the forward direction to $S_i'$.
\end{proof}

The duality of Lemma~\ref{thm:sorting-duality} and Corollary~\ref{thm:discreetness-duality} can be seen in Table~\ref{table:sorting-sequence}. If $f$ is in the list (and is in bold) for a sorting sequence $S_i$, then $25-f$ is in the list (and is in bold) for $S_i'$.

We also establish a lemma that provides necessary and sufficient conditions for when the sorting strategy is discreet.

\begin{lemma}\label{thm:sorting-discreetness}
The sorting strategy with sorting sequence $S_i$ can discreetly show that the number of fake coins is $f$ with $k$ coins in each pile if and only if there exist solutions of $f-p$ and $f$ fake coins with $k-1$ coins in each pile.
\end{lemma}
\begin{proof} 
The sorting strategy can discreetly show that the number of fake coins is $f$ with $k$ coins in each pile if and only if there is a configuration for $f$ fake coins in which there is a fake coin in every pile, and a configuration in which there is a real coin in every pile. In the first case, we remove a fake coin from every pile to obtain a configuration of $f-p$ fake coins with $k-1$ coins in each pile, and in the second case, we remove a real coin from every pile to obtain a configuration of $f$ coins with $k-1$ coins in each pile. Because the reduction can be reversed, the converse holds as well.
\end{proof}

\begin{example}
Consider the sorting sequence $(0, 1, 1, 2)$ with $k=5$. The sorting strategy can discreetly show that the number of fake coin is $10$. According to Lemma~\ref{thm:sorting-discreetness}, there should be configurations with $10$ and $6$ fake coins when there are $p(k-1)=16$ coins total. Figure~\ref{fig:discreetness} illustrates the reduction in the proof of Lemma~\ref{thm:sorting-discreetness}.

The sorting strategy can discreetly show that the number of fake coins is $10$ because there exist a configurations with $10$ fake coin where there is a fake coin in every pile, and one with a real coin in every pile. These two configurations are present on the left. Each pile is represented by a column of coins, and vertical lines separate piles with different weights. The fake coins are colored gray and the real coins white. By removing a row of fake or real coins, we obtain the configurations on the right.

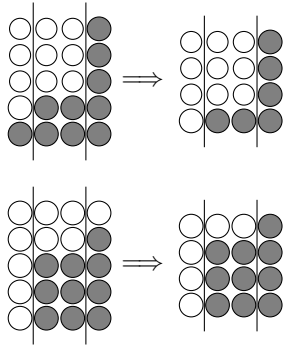
\begin{figure}[ht]
\centering
\begin{tikzpicture}[scale=0.35]
    \foreach \x in {1,...,4}
	    {\filldraw[fill=gray] (\x, 1) circle (0.45cm);}
    \foreach \x in {2,...,4}
        {\filldraw[fill=gray] (\x, 2) circle (0.45cm);}
    \foreach \y in {3,...,5}
        {\filldraw[fill=gray] (4, \y) circle (0.45cm);}
    
    \filldraw[fill=white] (1, 2) circle (0.45cm);
    \foreach \x in {1,...,3}
    {
        \filldraw[fill=white] (\x, 3) circle (0.4cm);
        \filldraw[fill=white] (\x, 4) circle (0.4cm);
        \filldraw[fill=white] (\x, 5) circle (0.4cm);
    }
	        
	\draw (1.5, 0) -- (1.5, 6);
	\draw (3.5, 0) -- (3.5, 6);
	\node[text width=1cm] at (6,3) {$\implies$};
	
	\foreach \x in {2,...,4}
	{
	    \filldraw[fill=gray] (\x, -6) circle (0.45cm);
	    \filldraw[fill=gray] (\x, -5) circle (0.45cm);
	    \filldraw[fill=gray] (\x, -4) circle (0.45cm);
	}
	\filldraw[fill=gray] (4, -3) circle (0.45cm);
    
    \foreach \y in {-6,...,-3}
	    {\filldraw[fill=white] (1, \y) circle (0.45cm);}
	\foreach \x in {2,3}
	    {\filldraw[fill=white] (\x, -3) circle (0.45cm);}
	\foreach \x in {1,...,4}
	    {\filldraw[fill=white] (\x, -2) circle (0.45cm);}
	        
	\draw (1.5, -1) -- (1.5, -7);
	\draw (3.5, -1) -- (3.5, -7);
	\node[text width=1cm] at (6,-4) {$\implies$};
	
    \foreach \x in {2,...,4}
        {\filldraw[fill=gray] (6.5+\x, 2-0.5) circle (0.45cm);}
    \foreach \y in {3,...,5}
        {\filldraw[fill=gray] (10.5, \y-0.5) circle (0.45cm);}
    
    \filldraw[fill=white] (7.5, 2-0.5) circle (0.45cm);
    \foreach \x in {1,...,3}
    {
        \filldraw[fill=white] (6.5+\x, 3-0.5) circle (0.4cm);
        \filldraw[fill=white] (6.5+\x, 4-0.5) circle (0.4cm);
        \filldraw[fill=white] (6.5+\x, 5-0.5) circle (0.4cm);
    }
	        
	\draw (8, 0.5) -- (8, 5.5);
	\draw (10, 0.5) -- (10, 5.5);
	
	\foreach \x in {2,...,4}
	{
	    \filldraw[fill=gray] (6.5+\x, -6+0.5) circle (0.45cm);
	    \filldraw[fill=gray] (6.5+\x, -5+0.5) circle (0.45cm);
	    \filldraw[fill=gray] (6.5+\x, -4+0.5) circle (0.45cm);
	}
	\filldraw[fill=gray] (10.5, -3+0.5) circle (0.45cm);
    
    \foreach \y in {-6,...,-3}
	    {\filldraw[fill=white] (7.5, \y+0.5) circle (0.45cm);}
	\foreach \x in {2,3}
	    {\filldraw[fill=white] (6.5+\x, -3+0.5) circle (0.45cm);}
	        
	\draw (8, -1.5) -- (8, -6.5);
	\draw (10, -1.5) -- (10, -6.5);
	
\end{tikzpicture}
\caption{An illustration of the reduction process in Lemma~\ref{thm:sorting-discreetness}.}
\label{fig:discreetness}
\end{figure}
\end{example}

\begin{corollary}
If the sorting strategy with sorting sequence $S_i$ can discreetly show that the number of fake coins is $f$ with $k$ coins in each pile, then there are solutions with $f-p$ and $f+p$ fake coins with $k$ coins in each pile.
\end{corollary}

\begin{proof}
If the sorting strategy can discreetly show that the number of fake coins is $f$, then by Lemma~\ref{thm:sorting-discreetness}, there exist configurations with $f-p$ and $f$ fake coins with $k-1$ coins in each pile. By respectively adding a fake coin or real coin to each pile of the two configurations, we obtain solutions with $f$ and $f+p$ fake coins with $k$ coins in each pile.
\end{proof}

Notice than in Table~\ref{table:sorting-sequence}, every bold number $f$ has $f-5$ and $f+5$ in the same row. The converse, however, is not true. In Table~\ref{table:sorting-sequence}, for the sorting sequence $(0,1,1,1,2)$ values $12$ and $13$ are not discreet, while $12-5$ and $12 + 5$ as well as $13-5$ and $13+5$ are present in the same row.

Now we are ready to calculate the smallest and largest number of possible fake coins for a sorting sequence $S_i$. We denote by $F_{\min}$ the minimum value of $f$ over all general solutions $(f_1, \dotsc, f_r, f)$ of~\eqref{partition-equation} for $S_i$. If there are enough total coins, $F_{\min}$ will be the minimum possible number of fake coins. 
The corresponding value for $S_i'$ is denoted $F_{\min}'$.

\begin{theorem}\label{thm:sorting-minmax}
For the sorting sequence $S_i$, $F_{\min} = \sum_{i=1}^r (i-1) p_i = \sum_{i=2}^r \sum_{m=i}^r p_m$. If $k < r-1$ then there are no solutions, and if $k \geq r-1$, the minimum possible number of fake coins is $F_{\min}$ and the maximum number is $pk-F_{\min}'$.
\end{theorem}

\begin{proof}
The value of $f$ in~\eqref{partition-equation} is minimized when each $f_i$ is as small as possible, which is when $f_i = i-1$. In this case, $f= \sum_{i=1}^r (i-1) p_i = \sum_{i=2}^r \sum_{m=i}^r p_m$.

If $k < r-1$, then no solutions satisfy the height bound, and $S_i$ cannot describe the weight relations for an outcome of the sorting strategy.

When $k\geq r-1$, then the solution with $f_i=i-1$ and $f=F_{\min}$ is realizable. Consequently, by Lemma~\ref{thm:sorting-duality}, there exist solutions satisfying the height bound for $S_i'$. Moreover, by Lemma~\ref{thm:sorting-duality}, a solution with maximum number of fake coins for $S_i$ is dual to a solution with the minimum achievable number of fake coins for $S_i'$. Consequently, the maximum is $pk-F_{\min}'$.
\end{proof}

Due to the duality between $S_i$ and $S_i'$, we have the following technical lemma:

\begin{lemma}
The following equality is true:
\[F_{\min} + F'_{\min} = p(r-1).\]
\end{lemma}

\begin{proof}
By Theorem~\ref{thm:sorting-minmax}, $F_{\min} = \sum_{i=1}^r (i-1) p_i$ and $F_{\min}' = \sum_{i=1}^r (i-1) p_{r-i+1} = \sum_{i=1}^r (r-1)p_i$, and consequently, $F_{\min} + F_{\min}' = p(r-1)$.
\end{proof}

We can provide a necessary condition on $k$ for discreetness and calculate minimum and maximum values of $f$ for which the strategy is discreet.

\begin{theorem}\label{thm:sorting-minmax-discreet}
Given a sorting sequence $S_i$, if $k < r+1$, then the sorting strategy is not discreet for any values, and for $k \geq r+1$, the minimum number of fake coins for which it is discreet is $p+F_{\min}$ and the maximum number is $p(k-1)-F_{\min}'$.
\end{theorem}

\begin{proof}
Given the sorting sequence $S_i$, let $f$ be the minimal value for which the sorting strategy is discreet that satisfies the height bound for $k$. There must exist a configuration for $f$ with a fake coin in every pile and a configuration with a real coin in every pile. The first configuration implies that $f \geq p + F_{\min}$. Observe that every configuration with $p + F_{\min}$ fake coins must necessarily have at least $r$ coins in the pile marked $r$. The second configuration requires that at least one of these configurations also have a real coin in every pile, including the last one, and thus $k \geq r+1$.
\end{proof}

Although these arguments suffice to calculate the minimum and maximum numbers of fake coins, in the following section we  analyze~\eqref{partition-equation} and the requirement that $f_r \leq k$.

\section{Duality and the Frobenius problem}\label{sec:duality-frobenius}

We analyze the solutions to~\eqref{partition-equation} and the requirement that $f_r \leq k$: By a change of variable, we can re-express~\eqref{partition-equation} as another Diophantine system of equations. If we let $x_r = f_1$ and $x_i = f_{r-i+1} - f_{r-i}-1$ for $i = 1, \dotsc, r-1$, then we re-express \eqref{partition-equation} and $f_r \leq k$ as

\begin{align}
&\sum_{i=1}^r\left(\sum_{m=r-i+1}^r p_m \right) x_i = f-\sum_{i=2}^r \sum_{m=i}^r p_m = f - F_{\min}\ \text{ \ with \ } \ x_1, \dotsc, x_r \geq 0, \label{nondecreasing-cond} \\
&\sum_{i=1}^rx_i \leq k-r+1. \label{height-bound}
\end{align}

By denoting $a_i = \sum_{m=r-i+1}^r p_m$ and $n = f-F_{\min}$, we get an equation 
\[a_1x_1+\dotsb+a_rx_r = n,\]
where $a_{i+1} - a_i = p_{r-i} > 0$.

There is a bijection between solutions to~\eqref{partition-equation} and~\eqref{nondecreasing-cond} given by the previous change of variable. We can recover solutions to~\eqref{partition-equation} from the solutions of~\eqref{nondecreasing-cond} by setting $f_i = i-1+ \sum_{j=r-i+1}^r x_j$. A solution $(x_1, \dotsc, x_r)$ to~\eqref{nondecreasing-cond} is equivalent to a \emph{general solution} $(f_1, \dotsc, f_r)$ for the sorting sequence, and one that additionally satisfies~\eqref{height-bound} is equivalent to a solution that satisfies the \emph{height bound} for $k$. 

As proven in Lemma~\ref{thm:sorting-duality}, a general solution $(f_1, \dotsc, f_r)$ satisfying the height bound for $k$ naturally corresponds to a general solution $(f_1', \dotsc f_r')$ for $S_i'$ satisfying the height bound with $pk-f$ fake coins and $f$ real coins through the relation $f_i'=k-f_{r-i+1}$. The equations for the general solution for the sorting sequence $S_i'$ are, if we let $x_i'=f_{r-i+1}'-f_{r-i}'-1$ and $x_r'=f_1'$,

\begin{align}
& \sum_{i=1}^r \left ( \sum_{m=1}^i p_m \right ) x_i' = (pk-f) - \sum_{i=2}^r \sum_{m=1}^{r-i+1} p_m = (pk-f) - F'_{\min}\ \text{ \ with \ } \ x_1', \dotsc, x_r' \geq 0, \label{dual-nondecreasing-cond} \\
&\sum_{i=1}^rx_i' \leq k-r+1. \label{dual-height-bound}
\end{align}

A solution $(x_1', \dotsc, x_r')$ to~\eqref{dual-nondecreasing-cond} corresponds to a general solution for $S_i'$ and one that satisfies~\eqref{dual-height-bound} corresponds to a solution satisfying the height bound for $k$.

The correspondence between solutions $(f_1, \dotsc, f_r)$ that satisfy the height bound and solutions $(f_1', \dotsc, f_r')$ creates the following expressions for  $x_i'$ in terms of $x_i$:
\[
x_i'= f_{r-i+1}' - f_{r-i}'-1 = (k-f_i) - (k-f_{i+1}) - 1 = f_{i+1} - f_i - 1 = x_{r-i},
\]
for $i < r$, and 
\[
x_r' = f_1' = k - f_r = k - \sum_{m=1}^r x_m. 
\]

This duality provides a necessary condition for the existence of a configuration with $f$ fake coins.

\begin{prop}\label{thm:necessary}
If a configuration of $f$ fake coins consistent with the sorting sequence $S_i$ exists, then there exists a general solution for $f$ with sorting sequence $S_i$ and a general solution for $pk-f$ for sorting sequence $S_i'$.
\end{prop}

The converse is not true. It is possible for there to exist general solutions for $f$ with sorting sequence $S_i$ and general solutions for $pk-f$ for sorting sequence $S_i'$ without there being a solution respecting the height bound as well.

\begin{example}
Consider the sorting sequence $(0,1,1,2,3)$ with $k=4$. The general solutions for $f=10$ are: $(f_1, f_2, f_3, f_4)=\{(0,1,2,6), (0,1,3,5)\}$. Note that none satisfy the height bound. For the reverse sorting sequence $(0,1,2,2,3)$, there is only one general solution: $(f_1, f_2, f_3, f_4)=(0,1,2,5)$, which does not satisfy the height bound either.
\end{example}

Though the converse of Proposition~\ref{thm:necessary} is not true, in Section~\ref{sec:redistribution} we show that it becomes true for large enough pile sizes and we estimate the $k$ for which the converse begins to be true.

\subsection{The Frobenius problem}\label{sec:frobenius}
To analyze the existence of solutions to~\eqref{nondecreasing-cond}, we will employ the theory of the \emph{Frobenius coin problem}. The Frobenius coin problem concerns the existence and nature of non-negative solutions $x_1, \dotsc, x_r$ to the Diophantine equation
\begin{equation} \label{frobenius}
a_1x_1+\dotsb+a_rx_r = n
\end{equation} 
given positive integers $a_1<a_2 < \cdots< a_r$ and non-negative $n$. Notice that equation~\eqref{nondecreasing-cond} is of the form of~\eqref{frobenius} with $a_i = \sum_{m=r-i+1}^r p_m$ and $n = f-F_{\min}$.

A fundamental result regarding the Frobenius coin problem is the following:

\begin{lemma}{\cite[Theorem 1.0.1]{RamirezAlfonsin}}\label{thm:existence-frob-number}
If $\gcd(a_1, \dotsc, a_r)=1$, then for all sufficiently large $n$, the equation~\eqref{frobenius} has non-negative integer solutions.
\end{lemma}

As a result, given integers $a_1, \dotsc, a_r$, there is some largest $n$ for which~\eqref{frobenius} has no solution. This number is called the \emph{Frobenius number} of $a_1, \dotsc, a_r$, and is denoted $g(a_1, \dotsc, a_r)$. Its calculation is the most important aspect of the Frobenius problem. In general, there exist no closed-form expressions for the Frobenius number, although there exist algorithms as well as exact formulae for some specific cases. A variety of upper bounds have been proven.

According to Brauer, Schur proved the following bound in his 1935 lectures~\cite{Brauer}:

\begin{align}\label{schurbound}
g(a_1, \dotsc, a_r) &\leq (a_1-1)(a_r-1)-1.
\end{align}

Other bounds in the spirit of~\eqref{schurbound} have also been found. In 1972, Erd\H{o}s and Graham~\cite{ErdosGraham} obtained~\eqref{erdosgrahambound}, and Selmer~\cite{Selmer} proved~\eqref{selmerbound} in 1977.
\begin{align}
g(a_1, \dotsc, a_r) &\leq 2a_{r-1} \left \lfloor \frac{a_r}{r} \right \rfloor - a_r, \label{erdosgrahambound} \\
g(a_1, \dotsc, a_r) &\leq 2a_r \left \lfloor \frac{a_1}{r} \right \rfloor - a_1 \label{selmerbound} 
\end{align}

In general, the two latter bounds are an improvement over~\eqref{schurbound}. Any of the bounds though, is superior in a few cases, depending on the precise values of $a_1, a_{r-1}, a_r$. These and other bounds can be found in \cite[Chapter~3.1]{RamirezAlfonsin}.

Additional facets of the problem, such as the calculation of the number of solutions to~\eqref{frobenius}, have also been studied. A definitive reference on the Frobenius problem itself as well as its many applications in pure mathematics and computer science is the monograph of Ramir\'{e}z Alfons\'{i}n~\cite{RamirezAlfonsin}.

Our strategy in Section~\ref{sec:arith-progressions} will be to apply both Lemma~\ref{thm:existence-frob-number} about the Frobenius problem as well as the results of Section~\ref{sec:redistribution} to equation~\eqref{nondecreasing-cond} in order to prove the existence of arithmetic progressions for the sorting strategy.

\section{Redistribution and the height bound}\label{sec:redistribution}

The goal of this section is to show that given a general solution for our sorting sequence, there exists a $K$ on the order of $p$ such that for any pile size $k \geq K$, there exists a solution that satisfies the height bound. We will do so via a redistribution procedure for the fake coins.

Note that if we have a sorting sequence, $f_i$ is the number of fake coins in the piles marked $i-1$.
\newline

\noindent \textbf{Redistribution procedure.} Let $\gcd(p_1, \dotsc, p_r) = c$. Suppose $f_{i+1} > f_i + (p_i + p_{i+1})/c$. We remove $p_i/c$ fake coins from each pile in the group marked $i$, and add $p_{i+1}/c$ coins to each pile in the group marked $i-1$.

Notice that $(p_i p_{i+1})/c$ coins that are removed from the  group marked $i$ and exactly $(p_i p_{i+1})/c$ coins are added to the group marked $i-1$.

\begin{lemma}
If the redistribution procedure is applied to a general solution, then the resulting coin configuration will also be a general solution for the same sorting sequence.
\end{lemma}

\begin{proof}
The number of fake coins in each pile in the group marked $i$ after the redistribution is $f_{i+1}-p_i/c$, and in the group marked $i-1$ is $f_i + p_{i+1}/c$. By the assumptions on when the redistribution procedure can be carried out, and the fact that no other piles are effected, all the pile relations remain true.
\end{proof}

We perform the redistribution procedure as many times as we can find a pair of values $f_i$ and $f_{i+1}$ that allow us to redistribute. The following lemma shows that this process will terminate.

\begin{lemma}
The redistribution of coins can only be carried out a finite number of times.
\end{lemma}
\begin{proof}
Let us define the quantity $M=p_1 f_1 +2p_2 f_2+\dotsb+r p_r f_r$. After a redistribution of coins from each pile marked $i$ to those marked $i-1$ the new quantity becomes:
\[p_1f_1+\dotsb+ip_i(f_i+p_{i+1}/c)+(i+1)p_{i+1}(f_{i+1}-p_i/c)+\dotsb+rp_rf_r = M - (p_ip_{i+1})/c,\] 
so that $M$ always decreases. However, note that $M$ cannot be negative, and thus, only a finite number of cycles of redistribution can be carried out.
\end{proof}

The redistribution allows us to reduce the difference between the number of fake coins in the lightest and heaviest piles.

\begin{prop}\label{thm:mindifferenceinpile}
If a general solution for a given sorting sequence exists, then a general solution such that $f_r - f_1 \leq (2p - p_1 - p_r)/c$ also exists.
\end{prop}
\begin{proof}
Let us perform the redistribution procedure until we are unable to do so. In this case, $f_{i+1}-f_i \leq (p_i + p_{i+1})/c$ for every $1 \leq i \leq r$. If we sum these expressions, we obtain the result.
\end{proof}

This establishes that the difference between the largest and smallest piles can be made close to $2p/c$.

\begin{prop}\label{thm:maxcoinsinpile}
Given a sorting sequence, if there exists a general solution for $f$ fake coins, then there exists a general solution where no pile has more than $(2p-p_1-p_r)/c+f/p-F_{\min}/p$ coins.
\end{prop}

\begin{proof}
From Proposition~\ref{thm:mindifferenceinpile} we know that $f_r \leq (2p-p_1-p_r)/c + f_1$. Then, because $f \geq  f_1p + p_2 + 2p_3 + \dotsb + (r-1)p_r = f_1p + F_{\min}$, the number of coins in the first pile, $f_1$, is less than or equal to $f/p-F_{\min}/p$. Thus we have that $f_r \leq (2p-p_1-p_r)/c+f/p-F_{\min}/p$.
\end{proof}

We have proven that we can satisfy a height bound on the order of $2p/c+f/p$. Now we can show a value of $f$ for which the height bound does not present an additional obstacle.

\begin{prop}\label{thm:p-heightbound}
Given a sorting sequence, if there exists a general solution for $f \leq pk-(2p^2-p_1p-p_rp)/c+F_{\min}$, then there exists a general solution respecting the height bound.
\end{prop}
\begin{proof}
If there exists a general solution for $f \leq pk-(2p^2-p_1p-p_rp)/c+F_{\min}$, then by Proposition~\ref{thm:maxcoinsinpile}, there exists a solution where no pile has more than $(2p-p_1-p_r)/c+f/p-F_{\min}/p \leq k$ coins.
\end{proof}

This shows that the height bound for $k$ is satisfied for all solutions when $f$ is less than a value on the order of $pk-2p^2/c$.

Now we are ready to state when the height bound does not interfere, or for which $k$ the converse of Proposition~\ref{thm:necessary} is true.

\begin{theorem}\label{thm:general-to-height-bound}
For $k \geq (4p-2p_1-2p_r)/c+1-r$ if there exists a general solution for $f$ with sorting sequence $S_i$ and a general solution for $pk-f$ with sorting sequence $S_i'$, then there exist solutions satisfying the height bound for both sorting sequences and the respective values.
\end{theorem}

\begin{proof}
By Proposition~\ref{thm:p-heightbound}, for $f \leq pk-(2p^2-p_1p-p_rp)/c+F_{\min}$, every general solution can be converted into a solution respecting the height bound.

If we apply same logic to the dual equation, that is, to the sorting sequence $S_i'$, we find that for $f' \leq pk-(2p^2-p_1p-p_rp)/c+F'_{\min}$ every general solution for the reversed equation can be converted into a solution respecting the height bound. Then, by the duality of $S_i$ and $S_i'$, this implies that for $(2p^2-p_1p-p_rp)/c-F'_{\min} \leq f$, if there exists a general solution for $f'=pk-f$ for the reversed equation, then there exists a solution respecting the height bound.

If $k \geq (4p-2p_1-2p_r)/c+1-r$, then
\[
pk \geq  (4p^2 -2p_1p -2p_rp)/c - (r-1)p = (4p^2 -2p_1p -2p_rp)/c-F_{\min} - F'_{\min},
\]
and consequently,
\[
(2p^2-p_1p-p_rp)/c-F'_{\min} \leq pk-(2p^2-p_1p-p_rp)/c+F_{\min}.
\]

Thus, any $0 \leq f \leq pk$ satisfies at least one of the equations: $f \leq pk-(2p^2-p_1p-p_rp)/c+F_{\min}$ or $(2p^2-p_1p-p_rp)/c-F'_{\min} \leq f$. That means if a general solution exists for both $f$ and $pk-f$, at least one of them can be converted to a solution satisfying the height bound. By duality, the other equation also has a solution satisfying the height bound.
\end{proof}

Numerical experiments for small values of $p$ indicate that the  conclusion of Theorem~\ref{thm:general-to-height-bound} may be true with a lesser requirement on $k$. Table~\ref{table:conjecture} lists all instances for $p \leq 6$ in which general solutions for $f$ exist for a sorting sequence as well as general solutions for $pk-f$ for the reverse sorting sequence, but there are no solutions respecting the height bound.

\begin{table}[ht]
\centering
\begin{tabular}{cccc}
$p$ & $k$ & Sorting sequence & Values of $f$ \\ \hline
$4$ & $3$ & $(0,1,1,2)$      & $6$ \\
$5$ & $3$ & $(0,1,1,1,2)$    & $7, 8$ \\
$5$ & $4$ & $(0,1,1,1,2)$    & $8, 12$ \\
$5$ & $4$ & $(0,1,1,2,3)$    & $10$ \\
$5$ & $4$ & $(0,1,2,2,3)$    & $10$ \\
$6$ & $3$ & $(0,0,1,1,1,2)$  & $7$ \\
$6$ & $3$ & $(0,1,1,1,1,2)$  & $8, 9, 10$ \\
$6$ & $3$ & $(0,1,1,1,2,2)$  & $11$ \\
$6$ & $4$ & $(0,0,1,1,1,2)$  & $8$ \\
$6$ & $4$ & $(0,0,1,2,2,3)$  & $10$ \\
$6$ & $4$ & $(0,1,1,1,1,2)$  & $9, 10, 14, 15$ \\
$6$ & $4$ & $(0,1,1,1,2,2)$  & $16$ \\
$6$ & $4$ & $(0,1,1,1,2,3)$  & $11, 12$ \\
$6$ & $4$ & $(0,1,1,2,2,3)$  & $11, 13$ \\
$6$ & $4$ & $(0,1,2,2,2,3)$  & $12, 13$ \\
$6$ & $4$ & $(0,1,1,2,3,3)$  & $14$ \\
$6$ & $5$ & $(0,1,1,1,1,2)$  & $10, 15, 20$ \\
$6$ & $5$ & $(0,1,1,1,2,3)$  & $17$ \\
$6$ & $5$ & $(0,1,1,2,3,4)$  & $15$ \\
$6$ & $5$ & $(0,1,2,2,2,3)$  & $13$ \\
$6$ & $5$ & $(0,1,2,2,3,4)$  & $15$ \\
$6$ & $5$ & $(0,1,2,3,3,4)$  & $15$ 
\end{tabular}
\caption{Instances for which there exists a general solutions of $f$ for $S_i$ and a general solution of $pk-f$ for $S_i'$, but not a solution satisfying the height bound, for $p \leq 6$.}
\label{table:conjecture}
\end{table}

Note than in Table~\ref{table:conjecture}, all these exceptions have $p/c>k$. Thus, it may be enough to require $k \geq p/c$. We have computationally verified that this is true up to $p=8$ by checking all sorting sequences for the nonexistence of counterexamples for all $k$ up to $(4p-2p_1-2p_r)/c +1-r$, the value provided by Theorem~\ref{thm:general-to-height-bound}.

\section{The existence of arithmetic progressions for the sorting strategy}\label{sec:arith-progressions}

In the previous section, we established bounds on $f$ for which the existence of a general solution implies the existence of a solution respecting the height bound for $k$. We utilize these results as well as the existence of Frobenius numbers to prove that the sorting strategy can show that there are $f$ coins for all values in an arithmetic progression.

\begin{lemma}\label{thm:upward-range}
Given a sorting sequence with $\gcd(p_1, \dotsc, p_r) = c$, there exists a minimal $\gamma$, divisible by $c$, such that for all $f> \gamma$ that are divisible by $c$, there exists a general solution for $f$.
\end{lemma}

\begin{proof}
If $\gcd(p_1, \dotsc, p_r)=1$, then $\gcd\{a_i=\sum_{m=r-i+1}^r p_m \mid i = 1, \dotsc, r\} = 1$, for if $d \neq 1$ were a divisor of all the numbers in the latter set, then $d$ would divide all $a_{r-i+1}-a_{r-i} = p_i$. Utilizing Lemma~\ref{thm:existence-frob-number}, let $\beta$ be the Frobenius number for $a_1, \dotsc, a_r$. Now consider \[\gamma =  \beta+\sum_{i=2}^r \sum_{m=i}^r p_m = \beta + F_{\min}.\] For all $f > \gamma$, there exists a solution to~\eqref{nondecreasing-cond}, which is a general solution for $S_i$.

In the case where $c \neq 1$, we can consider the ``reduced sorting sequence" with $p_i'=p_i/c$. By the previous result, there is a minimal $\gamma_0$ such that for all $f'>\gamma_0$, there exists a general solution for the reduced sorting sequence. Each such solution corresponds to a solution with $f=cf'$ for the original sorting sequence, so that for all $f>c\gamma_0=\gamma$ such that $f$ is divisible by $c$, there exists a general solution.
\end{proof}

Lemma~\ref{thm:upward-range} establishes the existence of a unique value $\gamma$ associated with a sorting sequence $S_i$. In what follows, given a sorting sequence $S_i$, we let $\gamma$ be this value, and $\gamma'$ be the corresponding value for $S_i'$. Our previous results allow us to bound the sizes of the value $\gamma$ and the analogous quantity $\gamma'$ for the reverse sorting sequence.

\begin{lemma}\label{thm:gamma-bounds}
Given a sorting sequence $S_i$, $\gamma < (2p^2-p_1p-p_rp)/c-F_{\min}'$ and $\gamma' < (2p^2-p_1p-p_rp)/c-F_{\min}$.
\end{lemma}

\begin{proof}
Consider a large $k$ so that $(2p^2-p_1p-p_rp)/c-F_{\min}' < pk-\gamma'$. By Lemma~\ref{thm:upward-range} and Proposition~\ref{thm:p-heightbound}, a solution for $S_i'$ satisfying the height bound exists for any $pk-f$ such that $\gamma' < pk-f \leq pk-(2p^2-p_1p-p_rp)/c+F_{\min}'$ that is divisible by $c$. Then, by Lemma~\ref{thm:sorting-duality}, there exists a solution for $S_i$ satisfying the height bound for every $f$ divisible by $c$ such that  $(2p^2-p_1p-p_rp)/c-F_{\min}'\leq f < pk-\gamma'$. Because the upper limit tends to infinity as $k$ grows, we see that a general solution exists for any $f \geq (2p^2-p_1p-p_rp)/c-F_{\min}$ that is divisible by $c$. By the minimality of $\gamma$, it must be less than $(2p^2-p_1p-p_rp)/c-F_{\min}$.

The second inequality is obtained by applying the result to the reverse sorting sequence.
\end{proof}

\begin{theorem}\label{thm:range}
Given the sorting sequence $S_i$ with $\gcd(p_1, \dotsc, p_r) = c$, when $k \geq (4p-2p_1-2p_r)/c+1-r$, for all $f$ such that $\gamma < f < pk-\gamma'$ and that are divisible by $c$, there can be $f$ fake coins.
\end{theorem}

\begin{proof}
By the assumption on the size of $k$ and Lemma~\ref{thm:gamma-bounds},
\[
k \geq (4p-2p_1-2p_r)/c - (r-1) = \frac{(4p^2-2p_1p - 2p_rp)/c - (F_{\min} + F_{\min}')}{p} > \frac{\gamma + \gamma'}{p},
\]
and thus $pk-\gamma' > \gamma$. We know that for all $f$ divisible by $c$ such that $\gamma < f < pk-\gamma$, there exist general solutions for $f$ with sorting sequence $S_i$ and general solutions for $pk-f$ with sorting sequence $S_i'$, so then because of Theorem~\ref{thm:general-to-height-bound}, there exist solutions satisfying the height bound.
\end{proof}

\begin{theorem}\label{thm:range-discreet}
Given a sorting sequence such that $\gcd(p_1, \dotsc, p_r)=c$, there exist $\delta$ and $\delta'$ such that when $k \geq (4p-2p_1-2p_r)/c+2-r$, for all $f$ such that $\delta < f < pk-\delta'$ and that are divisible by $c$, the sorting strategy can discreetly prove that there are $f$ fake coins.
\end{theorem}
\begin{proof}
Because $k-1 \geq (4p-2p_1-2p_r)/c+1-r$, by Theorem~\ref{thm:range}, there exist $\gamma$ and $\gamma'$ such that for all $f$ divisible by $c$ such that $\gamma < f < p(k-1)-\gamma'$, there exist general solutions satisfying the height bound for $k-1$. If we set $\delta = \gamma + p$ and $\delta' = \gamma'+p$, then if $pk-\delta' > \delta$, by Lemma~\ref{thm:sorting-discreetness}, the sorting strategy can discreetly prove that for all $f$ such that $\delta < f < pk-\delta'$ and that are divisible by $c$, there can be $f$ fake coins with $k$ coins in each pile.
\end{proof}

Together, Theorems~\ref{thm:range}~and~\ref{thm:range-discreet} prove that number of fake coins may be any value within an arithmetic progression with common difference $c$, delimited by a fixed $\gamma$ (and in the case of discreetness, $\delta$) and $pk-\gamma'$ (and in the case of discreetness, $pk-\delta'$). Both arithmetic progressions grow linearly in $k$.

Furthermore, we have managed to describe the possible values for the number of fake coins in the sorting sequence entirely in terms of the Frobenius problem when the number of coins is sufficiently large. As a result of Lemma~\ref{thm:existence-frob-number}, we know that the solutions $n$ to~\eqref{frobenius} consist of a set of exceptional values below the Frobenius number along with all $n$ greater than the Frobenius number. We have shown that when $k$ is sufficiently large, the set of possible numbers of fake coins consists of three parts: first, a set of exceptional values below $\gamma$ which correspond to the exceptional solutions of~\eqref{nondecreasing-cond}; second, an arithmetic progression delimited by $\gamma$ and $pk-\gamma'$, corresponding to the range of values past the Frobenius numbers of equations~\eqref{nondecreasing-cond} and~\eqref{dual-nondecreasing-cond}; third, a set of exceptional values above $pk-\gamma'$ corresponding to the exceptional solutions of~\eqref{dual-nondecreasing-cond}. The values for which the strategy is discreet also inherit this three-part structure.

We illustrate there results with the following example:

\begin{example}
Consider the sorting sequence $(0,1,1,1)$. According to Theorem~\ref{thm:range}, when $k \geq 16 -1 -8 = 7$, for all $f$ such that $8 < f < 4k$, the sorting sequence can prove that there are $f$ fake coins. Moreover, according to Theorem~\ref{thm:range-discreet}, for $k \geq 9$, the sorting strategy is discreet for $f$ in the range $12 < f < 4k-4$ as $\delta=12$ and $\delta'=4$.

We trace the proofs of these two assertions. We have that $p_1 = 1$, $p_2=3$, $F_{\min}=3$, and $F_{\min}'=1$. The corresponding equations~\eqref{nondecreasing-cond} and~\eqref{dual-nondecreasing-cond} are $3x_1 + 4x_2 = n = f - F_{\min}$ and $x_1'+4x_2'=n'=(pk-f)-F_{\min}'$. The Frobenius numbers for the two equations are $g(3,4) =5$ and $g(1,4)=-1$. For the first equation, solutions exist for $n= 0, 3, 4$, and $n > 5$. For the second equation, there is a solution for every $n'>0$. Furthermore, $\gamma=8$ and $\gamma'=0$, with $pk-\gamma' > \gamma$, and thus for any $k \geq 7$, the possible values for $f$ are $\{3,6,7\} \cup \{f | 8 < f < 4k\}$.

Consequently, if $k\geq 8$, we have that when there are $k-1$ coins in each pile, the possible values for $f$ are $\{3,6,7\} \cup \{f | 8 < f < 4(k-1)\}$. Thus, by Lemma~\ref{thm:sorting-discreetness}, we have that for the following values of $f$, the strategy is discreet: $\{4, 10, 11\} \cup \{f | 12 < f < 4(k-1)\}$.

The previous reasoning can be verified with Table~\ref{table:ranges}, in which we list the values of $f$ for which there exists a solution and print in bold those for which the sorting strategy is discreet.

\begin{table}[ht]
\centering
\begin{tabular}{cc}
 $k$ & \text{Values of $f$} \\ \hline
 $7$ & $3, 6, \mathbf{7}, 9, \mathbf{10}, \mathbf{11}, 12, \mathbf{13}, \mathbf{14}, \mathbf{15}, \mathbf{16}, \mathbf{17}, \mathbf{18}, \mathbf{19}, \mathbf{20}, \mathbf{21}, \mathbf{22}, \mathbf{23}, 24, 25, 26, 27$ \\
 $8$ & $3, 6, \mathbf{7}, 9, \mathbf{10}, \mathbf{11}, 12, \mathbf{13}, \mathbf{14}, \mathbf{15}, \mathbf{16}, \mathbf{17}, \mathbf{18}, \mathbf{19}, \mathbf{20}, \mathbf{21}, \mathbf{22}, \mathbf{23}, \mathbf{24}, \mathbf{25}, \mathbf{26}, \mathbf{27},$ \\
 & $28, 29, 30, 31$ \\
 $9$ & $3, 6, \mathbf{7}, 9, \mathbf{10}, \mathbf{11}, 12, \mathbf{13}, \mathbf{14}, \mathbf{15}, \mathbf{16}, \mathbf{17}, \mathbf{18}, \mathbf{19}, \mathbf{20}, \mathbf{21}, \mathbf{22}, \mathbf{23}, \mathbf{24}, \mathbf{25}, \mathbf{26}, \mathbf{27},$ \\
 & $\mathbf{28}, \mathbf{29}, \mathbf{30}, \mathbf{31}, 32, 33, 34, 35$ \\
\end{tabular}
\caption{Values of $f$ for the sorting sequence $(0,1,1,1)$.} \label{table:ranges}
\end{table}
\end{example}

None of the previous theorems provide information about the sizes of $\gamma, \gamma', \delta,$ or $\delta'$. Unfortunately, a formula for $\gamma, \gamma', \delta,$ or $\delta'$ is precluded by the lack of a general expression for the Frobenius number noted in Section~\ref{sec:duality-frobenius}, since not only is~\eqref{nondecreasing-cond} a case of~\eqref{frobenius}, but the latter equation can always be written in the form of~\eqref{nondecreasing-cond}: Because the coefficients $a_1, \dotsc, a_r$ of~\eqref{frobenius} are increasing, by setting $p_i=a_{r-i+1}-a_{r-i}$ and $f=n + \sum_{i=1}^{r-1} a_i$, we can express~\eqref{frobenius} in the form of~\eqref{nondecreasing-cond}.

Nonetheless, we can apply the upper bounds for the Frobenius number provided in Subsection~\ref{sec:frobenius} to calculate three upper bounds for $\gamma$ and $\gamma'$. The bounds for $\gamma$ utilize the coefficients from~\eqref{nondecreasing-cond}, which are
\[
a_1 = p_r, \quad \quad a_{r-1} = \sum_{m=2}^r p_m = p - p_1, \quad \quad a_r = \sum_{m=1}^r p_m = p,
\]
and the bounds for $\gamma'$ utilize the coefficients from~\eqref{dual-nondecreasing-cond}, which are
\[
a_1 = p_1, \quad \quad a_{r-1} = \sum_{m=1}^{r-1} p_m = p - p_r, \quad \quad a_r = \sum_{m=1}^r p_m = p.
\]
the following bounds are the respective applications of~\eqref{schurbound}, \eqref{erdosgrahambound}, and~\eqref{selmerbound} for $\gamma$ and $\gamma'$:

\begin{align*}
\gamma &\leq (p_r-1)p -p_r + \sum_{i=2}^r \sum_{m=i}^r p_m & \gamma' & \leq (p_1-1)p - p_1 + \sum_{i-2}^r \sum_{m=1}^{r-i+1} p_m  \\
\gamma & \leq 2 \left\lfloor \frac{p}{r} \right\rfloor (p-p_1) -p + \sum_{i=2}^r \sum_{m=i}^r p_m & \gamma' & \leq 2 \left \lfloor \frac{p}{r} \right \rfloor (p-p_r) - p + \sum_{i=2}^r \sum_{m=1}^{r-i+1} p_m \\
\gamma &\leq 2\left\lfloor \frac{p_r}{r} \right\rfloor p -p_r + \sum_{i=2}^r \sum_{m=i}^r p_m & \gamma' & \leq 2\left \lfloor \frac{p_1}{r} \right \rfloor p -p_1 + \sum_{i-2}^{r-1} \sum_{m=1}^{r-i+1} p_m
\end{align*}

Theorems~\ref{thm:range}~and~\ref{thm:range-discreet} establish a significant limitation on the power of the sorting strategy to prove definitively to an observer that the number of fake coins is a specific value. As $k$ increases, so does the possible range of values, and furthermore, so does the range of values for which the strategy is discreet. Thus, the power of the sorting strategy to rule out possibilities beyond those that are not divisible by $\gcd(p_1, \dotsc, p_r)$ is limited to a set of extremal values.

\section{The revealing factor for the sorting strategy}\label{sec:revealing-factor}

To calculate the revealing factor in the case of the sorting strategy, we must obtain all the general solutions for a sorting sequence that satisfy the height bound. Let $\mathcal F$ be the set of all $\vec{f}=(f_1, \dotsc, f_r)$ that are solutions for a given sorting sequence. 

We note that in~\cite[Section 6]{DiacoPreprint}, Diaco described a class of weighing strategies in which piles are also divided into classes of equal weight, although the classes are not compared to each other. As a result, Diaco's expression for the revealing factor is similar to that in the following Proposition.

\begin{prop}
Given a sorting sequence $S_i$, if $\mathcal F$ is the set of all general solutions $\vec{f}=(f_1, \dotsc, f_r)$ for $f$ fake coins that satisfy the height bound for $k$, then the revealing factor for the sorting strategy with sorting sequence $S_i$ and $f$ fake coins is
\[
X = \frac{\binom{pk}{f}}{\sum_{\vec{f}\in \mathcal{F}}\prod_{i=1}^p \binom{k}{f_i}^{p_i}}.
\]
\end{prop}
\begin{proof}
The number of old possibilities is clearly $\binom{pk}{f}$. Now we calculate the new possibilities: For every $\vec{f}$, we have that for each $i$, within each of the $p_i$ piles that weigh the same, any $f_i$ of the $k$ coins may be fake. Thus, there are $\sum_{\vec{f}\in S}\prod_{i=1}^r \binom{k}{f_i}^{p_i}$ new possibilities.
\end{proof}

\begin{example}
For the sorting sequence $(0,1,2, \dotsc, p-1)$, the minimum number of fake coins is $(p-1)p/2$. In this case, there is only one solution, $\vec{f}=(0,1,2,\dotsc, p-1)$. Therefore, the revealing factor is \[ \frac{\dbinom{pk}{\binom{p}{2}}}{\prod_{i=1}^p \binom{k}{i}}.\]
\end{example}

In general, calculating all solutions in $\mathcal F$ needs to be done computationally for each specific sorting sequence and value of $f$. However, by employing the equations~\eqref{nondecreasing-cond} and~\eqref{height-bound}, existing methods for calculating solutions to the Frobenius problem can be utilized.

\section{Acknowledgements}
We are grateful to the MIT PRIMES program for supporting this research. We also acknowledge useful conversations with Nicholas Diaco.


\begin{thebibliography}{99}

\bibitem{Brauer} A.~Brauer, On a problem of partitions, \textit{Am. J. Math.} \textbf{64} (1942), 299--312.

\bibitem{Diaco-Khovanova} N.~Diaco and T.~Khovanova, Weighing Coins and Keeping Secrets, \href{http://arxiv.org/abs/1508.05052}{\tt arXiv:1508.05052 [math.HO]}, (2015).

\bibitem{Diaco-KhovanovaPaper} N.~Diaco and T.~Khovanova, Privacy and Counterfeit Coins, \textit{Math Intelligencer} \textbf{38}(3) (2016), 6--10. 

\bibitem{DiacoPreprint} N.~Diaco, Counting Counterfeit Coin: A New Coin Weighing Problem, \href{https://arxiv.org/abs/1606.04170}{\tt arXiv:1606.04170 [math.HO]}, (2016).

\bibitem{ErdosGraham} P.~Erd\H{o}s and R.~L.~Graham, On a linear diophantine problem of Frobenius, \textit{Acta Arithmetica} \textbf{21} (1972), 399--408.

\bibitem{Knop} K.~Knop, The mystery of fake coins, \textit{Matematika}, N8, 29--31, (2008), in Russian.

\bibitem{RamirezAlfonsin} J.~L.~Ram\'{i}rez Alfons\'{i}n, \textit{The Diophantine Frobenius Problem.} Oxford, England: Oxford University Press (2005).

\bibitem{Selmer} E.~S.~Selmer, On the linear diophantine Problem of Frobenius, \textit{J. Reine Angewandte Math.} \textbf{293/294}(1) (1977), 1--17.

\bibitem{ShapovalovProblem} Kolmogorov Math Tournaments, \href{http://cdoosh.ru/kolm/kolm.html}{\tt http://cdoosh.ru/kolm/kolm.html}.

\end{thebibliography}
\end{document}